\documentclass[12pt]{amsart}
\usepackage{a4}
\usepackage{graphicx}
\usepackage{amsmath,amssymb,amsthm}
\usepackage{geometry}                
\usepackage{epstopdf}
\title{Double bubbles for immiscible fluids in $\RRR n$}
\author{Gary R. Lawlor}

\theoremstyle{plain}
\newtheorem{thm}{Theorem}[section]
\newtheorem{lem}[thm]{Lemma}
\newtheorem{cor}[thm]{Corollary}
\newtheorem{prop}[thm]{Proposition}
\theoremstyle{definition}
\newtheorem{defn}[thm]{Definition}

     \def \RRR {\mathbb{R}^}          \def \vvv {{\bf v}}    \def \ppp {{\bf p}}

\begin{document}

\maketitle
We use a new approach that we call unification to prove that standard weighted double bubbles in $n$-dimensional Euclidean space minimize immiscible fluid surface energy, that is, surface area weighted by constants.  The result is new for weighted area, and also gives the simplest known proof to date of the  (unit weight) double bubble theorem \cite{hhs}, \cite{hmrr}, \cite{br}.

As part of the proof we introduce a striking new symmetry argument for showing that a minimizer must be a surface of revolution.

\section{Introduction}

The double bubble problem in $\RRR 3$ and its variants have been a focus of research since about 1989.  The first published proof for the minimization property of a multiple bubble was done by an undergraduate group advised by Frank Morgan and headed by Joel Foisy \cite{fo}.  They proved that a standard double bubble in the plane (two overlapping disks separated by a circular arc or a line segment, with all bounding arcs meeting at $120^\circ$ angles) has the least perimeter required to separately enclose two given amounts of area.

Hass, Hutchings and Schlafly \cite{hhs} proved in 1995 that the least surface area required to separately enclose two equal amounts of volume is achieved by the standard double bubble.

A beautiful combination of symmetry and variational arguments, including ingeniously crafted variations, culminated in the triumph of Hutchings, Morgan, Ritor\'e and Ros \cite{hmrr} as they proved that in $\RRR 3$, standard double bubbles of unequal volumes also minimize surface area.  This required analyzing all equilibrium double bubble surfaces of revolution and eliminating those that could not be minimizers because of instability of the equilibrium or because a fragment of a bubble was too small.

Students of Morgan soon extended the result to higher dimensions, an enterprise that culminated in Reichardt's \cite{br} proof of the double bubble theorem in all dimensions.

The paper \cite{hmrr} also addresses the question of immiscible fluids, and proves that for certain volumes and for nearly-unit weights on surface area, the minimizers are standard.  This is the result that we extend in the present paper to all volume pairs and all weights in $n$ dimensions:

\vskip 0.1in
\noindent{\bf Theorem 8.1.} \emph{ Standard weighted double bubbles in $\RRR n$ all minimize weighted surface area among piecewise smooth boundaries of pairs of open regions with prescribed volume.  These minimizers are unique.}
\vskip 0.22in

\section{Acknowledgments}

The author gained valuable insight and momentum at the Workshop on
isoperimetric problems, space-filling, and soap bubble geometry in Edinburgh, Scotland, in March of 2012, and would like to express gratitude to the hosting International Centre for Mathematical Sciences and to the workshop organizers.

The author would also like to thank the referee for many helpful suggestions.  

\section{A new symmetry argument}

We will give a new argument (Proposition \ref{sym}) for why a minimizing double bubble must be a surface of revolution.  This argument is more robust than the previously known proof given in \cite{hmrr} in that it requires less knowledge of regularity of the minimizer.  

We first demonstrate the idea by giving a new proof of the regular isoperimetric theorem in $\RRR n$.

\begin{thm} \label{isoper}
The unique minimizer of surface area for enclosing a given volume in $\RRR n$ is a round ball.
\end{thm}
\begin{proof}  The proof involves existence, bisection, angular stretch, and reflection.
The heart of the argument is the fact that an angular stretch multiplies volume by a factor at least as large as it multiplies surface area, and the inequality is strict except on a surface of revolution.

We will first do the proof in $\RRR 2$.  Given a quantity of area to be enclosed, there exists a region $C$ of minimum perimeter enclosing that area.  

Choose any orthonormal basis for $\RRR 2$.  Translate a line perpendicular to the first basis vector until it bisects the area of $C$.  Call the bisecting line the $x$ axis.

Both halves of $C$ must have precisely half the original perimeter; otherwise we could replace the larger-perimeter half with the reflection of the smaller, contradicting minimality of $C$.  Note also that none of the length of the perimeter can be contained in the bisecting line, since this length could be discarded before reflecting.

Choose one half of $C$, and translate a line perpendicular to the other basis vector until it bisects the area of that half of $C$. Call the bisecting line the $y$ axis.

Choose one of the resulting quarters of $C$ and call it $C'$.  Then $C'$ has a quarter of the area and, necessarily, a quarter of the original perimeter of $C$. 

Certainly $C'$ is connected; otherwise we could choose a component of $C'$ whose percentage of the perimeter of $C$ is no larger than its percentage of the area of $C$.  Dilation stretches area by a larger factor than perimeter, so we could obtain a better piece than $C'$, and reflect it twice to contradict the minimality of $C$.

Now here is the main point.  Suppose $C'$ is not a quarter circle centered at the origin.  Simply do an angular stretch, mapping $\theta$ to $2\theta$, as in Figure 1.

\begin{figure} \label{ellipse}
\includegraphics[scale=0.4]{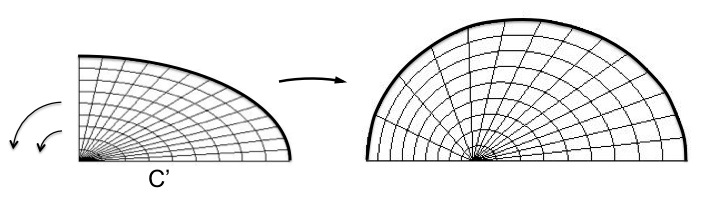}
\label{stretch}
\caption{}
\end{figure}
This doubles the area, but stretches the perimeter by \emph{less} than 2, since it is not all lined up with the $\theta$ direction.  Reflecting the stretched image across the $x$ axis then completes an enclosure that contradicts the minimality of $C$.

So $C'$ is a quarter circle.  The same argument can be applied to the other quarter of $C$ above the $x$ axis to show that it, too, is a quarter circle.  Of course, these quarter circles must line up since otherwise some of the perimeter would lie on the axis, which would lead to a contradiction as commented above.

Similarly the half of $C$ below the $x$ axis must be a half circle.  A priori, the $y$ axis selected to bisect the bottom half of $C$ might not be the same as the one that bisects the top half, but in the end, the upper and lower half circles must meet.  So $C$ is, indeed, a circle.

Now move to $\RRR n$.  The argument is similar; bisect a minimizer $M$ with a hyperplane, choose a half and bisect it with a perpendicular hyperplane, and continue until obtaining a piece $M'$ of $M$ in an orthant of $\RRR n$ having $1/2^n$ of the original volume and surface area.  One by one, do angular stretches (always by a factor of 2) on $M'$ until obtaining a region lying in a half space of $\RRR n$ and having half the volume of $M$.  Now unite this region with its reflection.  The result is an improvement on the original $M$ unless $M'$ was a piece of a round ball.  Similarly, each of $2^n$ pieces of $M$ must be spherical pieces and must line up, so that the original $M$ had to be a round ball. 

\end{proof}

\section{Unification}

We introduce here another key idea that we call \emph{unification}, in which we combine a family of problems, with their differing constraints, into a single minimization problem, with all surfaces competing together.  This placing of an optimization problem into a broader field makes it harder for a non-optimal competitor to have first variation zero.  Indeed, in the double bubble case we will show that there are \emph{no} competitors with unified first variation zero except the minimizers themselves: the standard weighted double bubbles.  This will eliminate the need for handling the more difficult questions of instability of equilibria as in \cite{hmrr}.

In order to allow constrained quantities to vary without losing the nature of the original question, we transfer the  constraints over to the measurement scheme.  We do this by dividing the measure of a competitor by the expected minimum measure for whatever constraints it satisfies.  We call this quotient the \emph{relative area} of the competitor surface.  

This levels the playing field and unifies whole families of minimization problems --- each with its conjectured minimizer --- into the single problem of seeing whether the relative area can ever be less than 1.  Constraints such as fixed volumes no longer keep the optimization problems segregated; when volumes change, this simply changes the denominator for the relative area calculation.

It is often helpful to reduce the set of competitors before unifying; in the present paper, for example, we will unify only after restricting our attention to surfaces of revolution.

\begin{defn} \label{uspace} (Unification space and relative area) Let $\{(Q_\alpha,T_\alpha,M_\alpha)\}$ be a collection of minimization problems.  For each $\alpha$, $T_\alpha$ denotes a set of competing objects (generically, ``surfaces'') that vie to minimize the quantity $Q_\alpha$ (such as weighted surface area), and $M_\alpha\in T_\alpha$ denotes a \emph{conjectured} minimizer.  

We select for each $\alpha$ a subset
$$\chi_\alpha\subseteq T_\alpha$$
containing $M_\alpha,$
and let 
$$\chi=\{(Q_\alpha,\chi_\alpha,M_\alpha)\}.$$

Then for any competitor $S\in \cup \chi_\alpha$, we define the \emph{relative area} $\mu(S)$ by finding the class $\chi_\alpha$ that $S$ belongs to and letting
$$\mu(S)=\frac{Q_\alpha(S)}{Q_\alpha(M_\alpha)}.$$

More generally, in the absence of conjectured minimizers for some constraint values, one might replace the denominator with a conjectured lower bound on the measurement of competitors satisfying those constraints.
\end{defn}

\section{Outline of the weighted double bubble proof}

Unification and the new symmetry argument will pave the way for the double bubble theorem to follow from an application of the Gauss map to the exterior of bubble clusters, amounting in essence to a calibration via the Gauss map.  Gauss map calibration was introduced by Kleiner \cite{kl} for proving isoperimetric inequalities in manifolds; see also \cite{hhm} for another application of this idea.

Kleiner's paper \cite{kl} notes (p. 38) an observation that he traces back to Almgren and others, that relates mean curvature to the derivative of surface area as volume changes in the isoperimetric profile for single bubbles.  Unification can be seen as applying this observation to multiple bubbles and extending it beyond mean curvature to include other quantities such as weighted area.

Unification appears promising for a wide variety of minimization problems, and can be used in connection with a number of other methods.  When applied to multiple bubble problems, one benefit of unification is to change average inequalities into piecewise inequalities.  A unified equilibrium surface that did better than the expected minimum would not only have smaller \emph{total} surface area than the corresponding standard double bubble, it must have smaller surface area for \emph{each piece} (the exterior of each bubble and the interface).  In addition, the mean curvature on each piece must be smaller.

These inequalities on corresponding pieces of a competitor versus the proposed minimizer open the way for the following simple argument. 

First, smaller mean curvature for a comparison surface means smaller Gauss curvature (or in higher dimensions, product of principal curvatures --- that is, the Jacobian of the Gauss map, also called Gauss-Kronecker curvature).  This is because the standard double bubbles consist of spherical caps, whose equal principal curvatures already give the best possible conversion factor between the sum and the product of mean curvatures.

Second, the image of the Gauss map (applied to the exterior surfaces of a competitor) has overlap(s) because of the outward bending at the singular meeting(s) of the two bubbles.  A nonstandard unified minimizer would, because of its disconnected singular set, have to have more overlap area than a standard double bubble; this is closely related to the isoperimetric theorem \emph{within} the sphere; indeed, we use the latter in our proof.

The greater overlap would require such a minimizer to generate more total Gauss image area in order to cover the sphere.  But it would have to achieve this with smaller-area domains and smaller Jacobians, which is impossible.

\section{Double bubbles in $\RRR n$}

\begin{defn} \label{swdb} A \emph{standard weighted double bubble} is made up of three distinct caps of $(n-1)$-spheres, all meeting along their common boundary $(n-2)$-sphere; the middle cap may be a flat $(n-1)$-disk. The angles between caps are related to the weights as in Figure 2; a triangle with sides perpendicular to the tangent planes at the junction must have side lengths proportional to the weights.  This is equivalent to the condition that the three conormal vectors (i.e., tangent to the surface and normal to the boundary) with lengths equal to the weights should sum to zero.

We allow three degenerate cases for standard weighted double bubbles.  The first involves disjoint spheres.  This corresponds to an interface weight being at least as large as the sum of the two exterior weights.   The second occurs when one prescribed volume is zero.  The third case involves one sphere inside the other, and occurs when the exterior weight of the enclosed sphere is at least as large as the sum of the other exterior weight and the interface.
\end{defn}

\begin{figure} 
\includegraphics[scale=0.4]{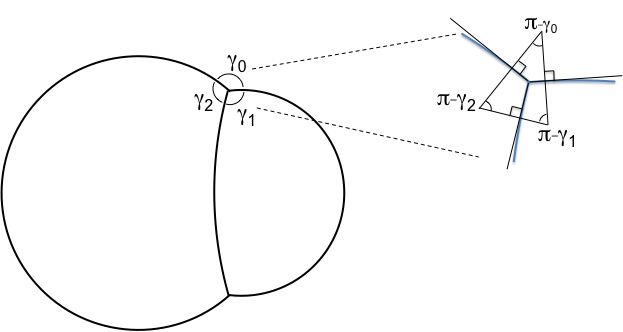}
\caption{}
\end{figure}

\begin{prop} \label{exists} For any pair of nonnegative volumes (not both zero) and nonnegative weights (not all zero) there is exactly one standard weighted double bubble with those volumes and weights.  Further, the radii and distance between centers in this bubble depend in a locally Lipschitz fashion upon the volumes and weights.
\end{prop}
\begin{proof} The degenerate cases, as described in Definition \ref{swdb}, are clear.    
Assume, then, that the weights satisfy the strict triangle inequality.  In particular, all three are positive; label them $w_0$ for the interface and $w_1,w_2$ for the respective exteriors.

Our proof builds on that of Proposition 2.1 in \cite{hmrr}. 

Consider a unit sphere through the origin as in Figure 3.  Form a small triangle with side lengths $w_0, w_1, w_2$, with the side of length $w_1$ being perpendicular to the tangent plane to the sphere at the origin and the side of length $w_2$ being exterior to the sphere.

Take another sphere intersecting the first at the origin (and elsewhere) at the correct angle so that the triangle's side of length $w_2$ is perpendicular to the tangent plane to the second sphere at the origin.

Find the unique third sphere that contains the intersection of the first two spheres and whose tangent plane at the origin is perpendicular to the triangle side of length $w_0$.  

Form a standard weighted double bubble from the caps of these spheres, with the bubble's interface being a subset of the third sphere. 

\begin{figure} 
\includegraphics[scale=0.4]{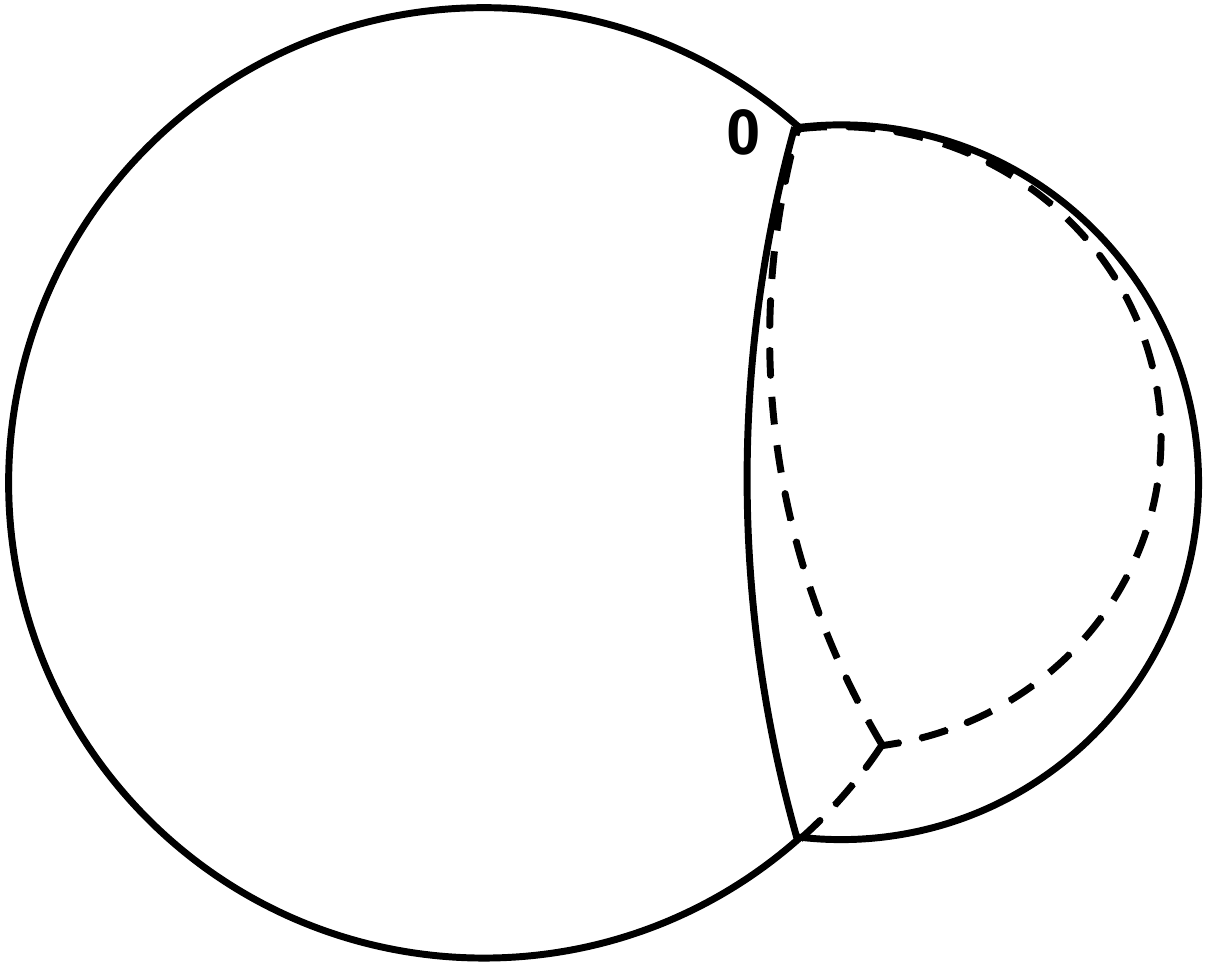}
\caption{}
\end{figure}

Now we will vary this figure.  With the angles held fixed, reducing 
the size of the second sphere increases the curvatures of both the second sphere and the interface, causing the volume of bubble 1 to increase and the volume of bubble 2 to decrease. Increasing the size of the second sphere has the opposite effect.  Thus,  we can adjust the second sphere until the ratio of volumes is correct, after which we dilate the picture to obtain the correct volume pair.  This gives existence and uniqueness.

For the Lipschitz dependence, suppose we have a standard weighted double bubble and we change slightly the prescribed volumes and/or weights.  First adjust the angles at which the spheres meet, to accommodate the new weights; this changes the volumes slightly.  To arrive at the new target volumes, as before adjust one radius until the volume ratio is correct, then dilate.  

It is geometrically clear that in the step where one radius is held fixed, the partial derivatives by the other radius of both volumes are nonzero.  The same is true in the dilation step.  One volume or the other will either increase on both steps or decrease on both steps,  which means that if either of the radius adjustments were large, then at least one volume change would have been large.  This contradiction completes the proof of the local Lipschitz dependence.

\end{proof}

We are now ready to define our unification space.

\begin{defn}[The unification space $\chi$]\label{un} For any pair of nonnegative volumes $(V_1,V_2)$ (not both zero) and nonnegative weights $w_0, w_1, w_2$ satisfying the triangle inequality, let $\alpha=(V_1,V_2,w_0,w_1,w_2)$ and define $\chi_\alpha$ to be the family of all
double bubble competitors that are piecewise-smooth surfaces of revolution (unions of $(n-2)$-spheres centered on a common line),  enclosing volumes $V_1,V_2$, with connected exterior region.  Let $\chi$ be the union of all such $\chi_\alpha$.  

Let $Q_\alpha(S)$ measure weighted surface area of a competitor $S$, with weights $w_1,w_2\in[0,1]$ on the exteriors of bubbles 1 and 2, respectively, and $w_0$ on the interface between them. 

Our conjectured minimizers $M_\alpha$ are the standard weighted double bubbles of Definition \ref{swdb}.

\end{defn}

To justify the assumption that the region of $\RRR n$ exterior to both bubbles is connected, suppose instead that some hollow were left unfilled by the surrounding bubbles.  The boundary of the hollow consists of interface with bubble 1 and interface with bubble 2; whichever has larger (unit-weight) surface area can be deleted, thus filling the hollow with the appropriate bubble.  This may increase the weight on the remaining interface, but because of the triangle inequality on the weights, the total weighted surface area will not increase.  Since this process would increase a volume, the proposed minimum weighted surface area would also increase, and relative area would decrease.

\vskip 0.1in

We need to discuss existence, symmetry and regularity of minimizers of $\mu$.

\begin{prop}
There exists a minimizer of relative area in the unification space $\chi$.
\end{prop}
\begin{proof}
Of course, if the infimum of $\mu$ is 1, then every proposed minimizer realizes the minimum relative area.  So suppose that the infimum of $\mu$ is $\mu_0<1$.

By Morgan's argument in \cite{begin}, sections 13.4 and 16.2, for any fixed volumes and weights satisfying the strict triangle inequality, there exists a double bubble that minimizes weighted surface area.  A priori this is an integral current but might not be piecewise smooth, but we will soon see that it is, in fact, piecewise smooth.

On the other hand, if the weights only satisfy the non-strict triangle inequality, then any piece of the expensive surface may be thought of as two membranes superimposed, representing the two cheaper surfaces.  Thus, by the regular isoperimetric inequality, we cannot do better than two separate spheres or one inside another, depending on whether the interface or one of the exteriors is the expensive surface.  

Scaling of volume pairs or of weight triples does not affect relative area, so we may restrict attention to the classes in which the larger volume equals 1 and the largest weight equals 1.  Then we get a minimum relative area function $f$ on a compact space of volume pairs and weights; if $f$ is continuous then there exists a minimizer of relative area.  To verify continuity, note first that for any two nearby sets of volume pairs and weights, the proposed minima of weighted areas are nearly equal.  Second, a small perturbation of the actual minimizer in either class gives a candidate for minimization in the other class; thus, the actual minimum is less than or nearly equal to the minimum in the other class.  But this works in both directions, so the actual minima are nearly equal.  Dividing by the nearly equal proposed minima, we obtain nearly equal relative areas, and we see that $f$ is continuous and there exists a $\mu$ minimizer $C_0$.

Now the weights that go with $C_0$ must satisfy the strict triangle inequality; otherwise $\mu(C_0)$ would equal 1.  From this we can deduce that $C_0$ has finitely many bubble components; otherwise a tiny component of volume $\epsilon$ could be either eliminated or merged with the other bubble.  One of these choices would remove at least half of the weighted surface area of the tiny component, which by the regular isoperimetric inequality is at least a constant times $\epsilon^{2/3}$.  But the corresponding adjustment in the proposed minimum surface area would only amount to $O(\epsilon)$, so that $\mu$ would be reduced, yielding a contradiction.

Now by the symmetry argument below, $C_0$ must in fact be a surface of revolution.  Each of the finitely many segments in the generating network of curves for $C_0$ must be a generalized Delaunay curve and thus smooth.  So $C_0$ is, indeed, piecewise smooth.  (Note also that this last argument also works for the minimizers within each class, so that all are piecewise smooth as claimed above.)

\end{proof}

For symmetry, we begin with a version of the ham sandwich theorem for double bubbles.

\begin{lem} \label{bisectors} Given two regions in $\RRR n$, every 2-plane in $\RRR n$ contains a vector normal to some hyperplane that divides the volumes of both regions in half.
\end{lem}
\begin{proof} 
Let $V_1$, $V_2$ be regions in $\RRR n$ and $P$ a 2-plane.  Choose a point $\ppp\in P$.  For each vector $\vvv\in P$ there is a hyperplane normal to $\vvv$ that bisects the volume of $V_1$.  If $V_1$ is not connected this hyperplane may not be unique; make it unique by choosing it to pass as near to $\ppp$ as possible.  Call the hyperplane $H(\vvv)$, and do the same for all vector directions $\vvv\in P$.

Let $f(\vvv)$ be the difference between the volume of $V_2$ lying on the positive side of $H(\vvv)$ (with respect to the direction of $\vvv$) and the volume of $V_2$ on the negative side.

Because of the uniqueness of $H(\vvv)$ it will follow that $f$ varies continuously with $\vvv$.  Also, $f(-\vvv)=-f(\vvv)$ for all $\vvv$, so by the intermediate value theorem, $f$ equals zero for some $\vvv\in P$.
\end{proof}

\begin{prop}[Symmetry of minimizers] \label{sym} In $\RRR n$ with $n>2$, an integral current that is a minimizer of weighted surface area enclosing two fixed volumes must be a surface of revolution, that is, a union of $(n-2)$-spheres centered at points of a fixed axis line $L$.  It follows that a minimizer of $\mu$ is also a surface of revolution.
\end{prop}
\begin{proof}
We follow the ideas of the proof of the isoperimetric Theorem \ref{isoper}.

Let $M$ be a double bubble competitor that minimizes weighted surface area for its given enclosed volumes. 
Using Lemma \ref{bisectors}, find a hyperplane $H_1$ that bisects both volumes, and a second hyperplane perpendicular to $H_2$ that bisects both half-volumes on one side of $H_1$, and so forth until obtaining a piece $M'$ of $M$ bounded by $n-1$ hyperplanes and having $1/2^{n-1}$ of each original volume (and, necessarily, the same fraction of the original surface area) of $M$.  

Let $L$ be the line formed by the intersection of the $n-1$ hyperplanes.

Iteratively apply angle-doubling stretches until matching half the volumes of $M$ on one side of a hyperplane.  Reflect across the last hyperplane to complete a figure matching the original volumes of $M$.  Now unless $M'$ was already a piece of a surface of revolution, this stretching will have multiplied surface area by less than it did volume, contradicting the minimality of $M$.

Each time when we chose one half or the other of a bisected pair of volumes, the choice was arbitrary.  So all of $M$ must be a union of $2^{n-1}$ pieces of surfaces of revolution.  If they did not match up, there would be surface area on the bisecting planes, leading to a contradiction as in the isoperimetric proof \ref{isoper}.  In the end, all of $M$ must be a surface of revolution.
\vskip 0.1in

Of course, any minimizer of $\mu$ is also a minimizer of weighted surface area for its given enclosed volumes, so it, too, is a surface of revolution, as required.
\end{proof}

Now that we know a minimizer is a surface of revolution, its regularity also follows:
\begin{prop}[Regularity of minimizers] \label{regu} The planar generating network of a minimizer of $\mu$ has finitely many junction points around which the exterior region and the two bubble regions meet.  The junction points are connected by smooth curve segments.
\end{prop}
\begin{proof}
If the triangle inequality on the surface weights is not strict, the isoperimetric inequality guarantees that a minimizer must consist of two spheres, either nested or disjoint.

Suppose then that the weights satisfy the strict triangle inequality.  Then there is a lower bound on the volume of a component of a minimizer, since a microscopic component could either be deleted (merged with the exterior) or merged with the opposite bubble, depending on which option reduced the (unweighted) surface area the most.  The isoperimetric inequality applied to the tiny component, together with the strictness of the triangle inequality on the weights, would then guarantee a sufficient decrease in weighted surface area to more than pay the tiny cost of adjusting the volumes in the proposed minimizer.

\begin{figure}[h]
\includegraphics[scale=0.5]{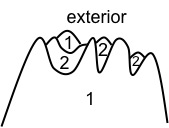}
\caption{}
\end{figure}

A single component could meet the exterior in many fingers, separated by little components of the opposite bubble.  So we might still worry about the finiteness of the entire structure.
But if we imagine building up the planar picture from scratch by adding one connected bubble component at a time,  since the exterior region must be connected, each addition of a bubble component adds at most two new junction points.  There are finitely many components, so there are finitely many junction points.

Between junction points, segments of curve are smooth, being governed by a differential equation arising from variational calculus.
\end{proof}

\begin{prop}[Singular structure of minimizers]  \label{existreg}  The smooth pieces of a minimizer of relative area meet in threes along $(n-2)$-spheres at angles matching those of the corresponding standard weighted double bubbles.\end{prop}
\begin{proof}
Let $S$ be a minimizer of $\mu$, which by Proposition \ref{sym} is a surface of revolution.  Let $M$ the standard double bubble with the same volumes and weights as $S$.  Let $\Sigma_S$ be the network of planar curves that generate $S$, and $\Sigma_M$ the generating network for $M$.  

It is a standard fact that these curves never meet in fours, since then two of the four separated regions would be pieces of the same region (or exterior) and the junction could either separate or pinch together and reduce the surface area they generate.

\begin{figure}[h]
\includegraphics[scale=0.5]{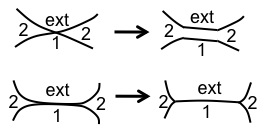}
\caption{}
\end{figure}

A consideration of forces pulling on the triple junction shows that the tangent vectors to the three curves meeting at a junction point, taken with lengths equal to the weights on the corresponding surface areas, must sum to zero.  This matches the condition in Definition \ref{swdb}.

\end{proof}

\section{Critical points of $\mu$}

An analysis of critical points in a unification space includes the behavior in both the (abstract) interior and boundary of that space.  

A point of the unification space boundary is, by definition, a set of prescribed volumes and weights that cannot be varied in all directions.  For example, if the weights satisfy only the non-strict triangle inequality, then there are only certain ways in which we can vary the weights from there and still maintain the required triangle inequality.  

The two types of boundary points are:

\begin{enumerate}
\item one of the prescribed volumes is zero, or
\item the weights only satisfy the non-strict triangle inequality.
\end{enumerate}

In all such cases the regular isoperimetric inequality shows that the minimizers have relative area 1.

\subsection{Requirements for an interior critical point of $\mu$}

A critical point must, of course, be in equilibrium with respect to the standard variational principles that apply to \emph{fixed} volumes and weights.  Thus, a critical point will consist of pieces of constant mean curvature that meet at angles prescribed according to the weights (see Definition \ref{swdb}).

\begin{prop} 
\label{lessarea} 
Let $\mu_0$ be the minimum value of the relative area $\mu$ on $\mathcal{K}$, and let $w_0, w_1, w_2$ and $V_1,V_2$ be weights and volumes of some minimizer $C_0$ that achieves relative area $\mu_0$.   Supposing that $\mu_0<1$, the surface $C_0$ will have to be a nonstandard constant mean curvature weighted double bubble of revolution.  Let $M_0$ be the associated standard double bubble with the same weights and volumes. 

Then the two exterior surface areas of $C_0$ must be less than or equal to $\mu_0$ times the corresponding surface areas of $M_0$. \end{prop} 

\begin{proof} In degenerate cases, $\mu_0=1$ and $C_0=M_0$.  

Suppose now that the theorem is false for a nondegenerate case; say the surface area of $C_0$'s bubble 1 exterior is greater than $\mu_0$ times that of $M_0$.  Now begin to reduce, at unit speed, the weight $w_1$.  
At the same time, move $M_0$ continuously through the space of proposed minimizers to continue to match the prescribed values of volume and weight.  The important point here is that $M_0$ has first variation zero in weighted area $Q$ with respect to changes that preserve volumes, so the initial rate of change of $Q(M_0)$ is due entirely to the change in $w_i$ and not in the shape of $M_0$.  

We see that the initial rates of change of $Q(M_0)$ and $Q(C_0)$ are equal to (minus) the respective areas of the exterior of bubble 1 in each.  This causes $\mu$ to dip below $\mu_0$, contradicting the assumption that we already had the minimizer of $\mu$.

\end{proof}
  
\begin{prop} \label{lesscurvature} With hypotheses as in the previous proposition, the (constant) mean curvature on each exterior piece of $C_0$ is less than or equal to $\mu_0$ times the corresponding mean curvature on $M_0$. \end{prop}

\begin{proof} The proof is similar to that of the previous proposition.  Supposing the proposition false, push inward slightly on any exterior having larger mean curvature than specified.  To first order this would change the surface areas of $M_0$ and $C_0$ by amounts that would enable a decrease in $\mu$, which was already at its minimum.
\end{proof}

\begin{cor} \label{bigsleeve} With the same hypotheses as above, the Gauss image of the exterior of either bubble of $C_0$ has less area than the Gauss image of the corresponding bubble exterior for $M_0$.
\end{cor}
\begin{proof}
 For fixed mean curvature, Gauss-Kronecker curvature (the product of principal curvatures) is largest when principal curvatures are equal, so that by Proposition \ref{lesscurvature}, the Gauss-Kronecker curvature at any regular point of $C_0$ is at most $\mu_0$ times the constant mean curvature on the corresponding spherical piece of $M_0$.  But the Gauss-Kronecker curvature is also the Jacobian of the Gauss map.  By Proposition \ref{lessarea}, the exterior pieces of $C_0$ have less area than the corresponding pieces of $M_0$, and with a smaller Jacobian, the areas of their Gauss images will also be smaller than for $M_0$.
\end{proof}

\section{Size of the Gauss map overlap due to the singular set}

The singularities in the exterior of an equilibrium double bubble cause an overlap in the union of images of the Gauss map applied to the smooth portions of the exterior.  We need to know that there is always more total overlap area for a nonstandard competitor than for a standard one.

\begin{defn} An \emph{antenna} of a competing surface of revolution is a vector bisecting the angle at a singularity, as in Figure 6.  In the example of the figure, there are one \emph{leftward-pointing} antenna and two \emph{rightward-pointing} antennae.  A vertical vector is considered as both leftward and rightward pointing.

The most important characteristic of an antenna will be its steepness, which we will also refer to as its height (independent of the position of the singular point where the vector starts).  

A \emph{sleeve} will refer to the Gauss image of a smooth exterior component of a competitor surface of revolution; an \emph{end sleeve} is a spherical cap.  

A \emph{cuff} will refer to the overlap annulus formed as the intersection between two consecutive sleeves.  By Proposition \ref{existreg}, cuffs have the same width, which is $\pi/3$ for unit surface weights.

The \emph{inner perimeter} of a cuff is the smaller of the surface areas of the two $(n-2)$-dimensional boundary spheres of the cuff.

We will call one cuff higher than another if the corresponding antenna is steeper; equivalently, the higher of two cuffs is the one with a larger inner perimeter.

\end{defn}

\begin{figure} \label{antennae}
\includegraphics[scale=0.6]{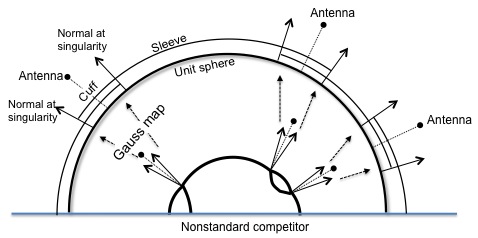}
\caption{Antennae, sleeves, and cuffs}
\end{figure}

Notice that a standard double bubble has two end sleeves that share one cuff.
All other serious competitors will have at least one extra double-cuffed sleeve in the middle.

We need the following technical lemma.
\begin{lem} \label{tech}
Let $\beta>0$, $c<d$, and $f(t)>0$ for $t\in(c,d+\beta)$.  Suppose that $f'(t)/f(t)$ is a decreasing function on $(c,d)$.  Then the function
$$h(t)=\frac{\int_{t}^{t+\beta}f(\tau)d\tau}{f(t)}$$
is also decreasing on $(c,d)$.
\end{lem}
\begin{proof}
Differentiating $h(t)$ we get
$$\frac{d h(t)}{dt}=\frac{f(t)\bigl(f(t+\beta)-f(t)\bigr)-f'(t)\int_{t}^{t+\beta}f(t)dt}{f(t)^2}.$$
Consider the numerator of the above expression.  To show that it is negative, first note that it would be zero if $\beta=0$.  Differentiating by $\beta$, we get
$$f(t)f'(t+\beta)-f'(t)f(t+\beta)=f(t+\beta)f(t)\Bigl(\frac{f'(t+\beta)}{f(t+\beta)}-\frac{f'(t)}{f(t)}\Bigr).$$
If $f'/f$ is decreasing, the expression will be negative, as desired.

\end{proof}

\begin{prop} \label{higherlower} A higher cuff has greater surface area, while a lower cuff has greater surface area \emph{per unit} of inner perimeter.
\end{prop}

\begin{proof} Set up spherical coordinates horizontally so as to agree with Figure 6, so that in spherical coordinates two cuffs are defined by $t_1 \le \phi \le u_1$ and $t_2\le \phi \le u_2$.  By reflecting each cuff, if necessary, we may assume that $t_i\le \pi-u_i$ for $i=1,2$.  Since the cuffs have the same width, we can define $\beta=u_1-t_1=u_2-t_2>0$.

Let  $\alpha_k$ denote the $k$-dimensional volume of the unit ball in $\RRR k$, and set
$$f(t)=(n-1)\alpha_{n-1}\sin^{n-1}(t).$$
Then the inner perimeters of the cuffs are $f(t_1)$ and $f(t_2)$, and the areas of the cuffs are
$$\int_{t_1}^{u_1} f(t)\,dt \mbox{ and }\int_{t_2}^{u_2} f(t)\,dt.$$
To see that the higher cuff has the larger surface area, note that
$$\frac{d}{ds}\int_s^{s+\beta} \sin^{n-1}(t)\,dt\,\,=\,\, \sin^{n-1}(s+\beta)-\sin^{n-1}(s),$$
 which is positive until $\pi-(s+\beta)=s$.
 
Lemma \ref{tech} shows that the lower cuff has greater surface area per unit of inner perimeter.
\end{proof}

We now prove the main proposition of this section, showing that a hypothetical minimizer that did better than standard would have to have more overlap area than the standard.

\begin{prop} \label{moreover} Suppose that $C_0$ is a minimizer of relative area (and thus a surface of revolution, by Proposition \ref{sym}) and that $C_0$ has less weighted surface area than $M_0$, the corresponding standard weighted double bubble.  Then $C_0$ must have a double-cuffed sleeve whose cuffs together have more area than the cuff of $M_0$.
\end{prop}
\begin{proof}
Let $Y$ be the larger of the two sleeves of $M_0$, and $K_M$ its cuff.  We will divide into two cases.

\noindent\emph{Case 1: } All antennae of $C_0$ point to the same side --- all leftward or all rightward.  Since the largest sleeve of $C_0$ is an end sleeve and by Corollary 6.3 must be smaller than $Y$, its cuff must be higher than $K_M$, so by Proposition \ref{higherlower}, we are finished. 

\noindent\emph{Case 2: } Antennae point both ways.  Consider any sleeve $\mathcal{S}_C$ having a leftward and a rightward antenna.  As before, if either cuff of $\mathcal{S}_C$ is at least as high as $K_M$, then it has at least as much surface area, and together with the other cuff we have strictly more area, as desired.

On the other hand, suppose that both cuffs of the sleeve are lower than $K_M$.  If we can show that the two cuffs together have more inner perimeter than does $K_M$, then by Proposition \ref{higherlower}, the total area of the two cuffs will be greater than the area of $K_M$.

By Proposition \ref{bigsleeve}, the sleeve $\mathcal{S}_C$ is smaller than $Y$, so its complement $\mathcal{S}_C'$ is larger than the complement $Y'$ on the unit sphere.  Further, since there is a great sphere disjoint from $\mathcal{S}_C'$, as we reduce the area of $\mathcal{S}_C'$ to match that of $Y'$ we also reduce its total perimeter.  Since $Y'$ is a spherical cap, by the isoperimetric theorem on the sphere, its perimeter is uniquely least for its given area.  We deduce that the perimeter of $Y'$ (and thus of $Y$) is less than that of $\mathcal{S}_C'$ (and thus of $\mathcal{S_C}$).  Now the perimeter of $\mathcal{S}_C$ is the total inner perimeter of its two cuffs, while the perimeter of $Y$ is the inner perimeter of its cuff $K_M$.  By Proposition \ref{higherlower}, the cuffs have more area than $K_M$, as required.
\end{proof}

\section{The main theorem}
\begin{thm} \label{mainthm} Standard weighted double bubbles in $\RRR n$ all minimize weighted surface area among piecewise smooth boundaries of pairs of open regions with prescribed volume.  These minimizers are unique. \end{thm}

\begin{proof}  Suppose some standard weighted double bubble (Definition \ref{swdb})  is not a minimizer. By Morgan's argument in \cite{begin}, sections 13.4 and 16.2, and by Propositions \ref{sym} and \ref{regu}, there must exist a piecewise smooth surface of revolution having the same volumes and smaller surface area than the standard.  Restrict attention, then, to piecewise smooth surfaces of revolution, and expand the space of competitors to $\chi$ (see Definitions \ref{uspace} and \ref{un}), thus allowing volumes and surface weights to vary.

By Proposition \ref{existreg}, within $\chi$ there exists a minimizer $C_0$ of relative area $\mu$ with $\mu(C_0)=\mu_0$ for some $\mu_0<1$.  Let $M_0$ be the standard weighted double bubble whose volumes and surface weights match those of $C_0$.
 
Now $C_0$ must be a critical point in the space $\chi$.  Boundary critical points have $\mu=1$, so $C_0$ is an interior critical point.

By Propositions \ref{lessarea} and \ref{lesscurvature}, the exterior surface of each bubble of $C_0$ has surface area and mean curvature at most $\mu_0$ times the corresponding value for $M_0$.  
 
By Corollary 6.3, the Gauss image of the exterior of each bubble of $C_0$ has less area than the Gauss image of the corresponding bubble exterior for $M_0$. Moreover, by Proposition \ref{moreover} there is more overlap in the Gauss image of $C_0$ than in the image of $M_0$.  Taken together, these facts would prevent the Gauss image of the exterior of $C_0$ from covering the sphere, a contradiction.
 
To prove uniqueness, now suppose that besides a standard minimizer $M_1$ with certain volumes and weights, there exists another minimizer $M_2$ with the same volumes and weights.  The surface areas and mean curvatures on exterior pieces of $M_2$ cannot exceed the corresponding quantities for $M_1$; otherwise a slight reduction in volume or weight would reduce $\mu$ below 1, which we now know cannot happen.  The singular set on the exterior cannot consist of more than a single $(n-2)$-sphere or by Proposition \ref{moreover} there would be too much Gauss image overlap, and the surface areas and mean curvatures on the exterior of $M_2$ must, in fact, equal those of $M_1$ or the Gauss map on the exterior of $M_2$ could not cover the sphere.   

The exterior of $M_2$ must now be identical to that of $M_1$.  Then the interface surface area and curvature must also be the same in $M_2$ as $M_1$, so that $M_2$ is forced to use the most economical way to connect its singular sphere to the axis, namely the same way that $M_1$ does so.
\end{proof}

 \end{document}